\documentclass[12pt,a4paper]{amsart}
\usepackage{amssymb}
\usepackage{amsmath}
\usepackage{mathtools} 
\usepackage{stackrel} 
\usepackage[hidelinks]{hyperref}
\usepackage{etoolbox}
\usepackage{enumitem}
\usepackage{xcolor}
\usepackage{tikz-cd}
\usepackage[normalem]{ulem}
\hypersetup{
	colorlinks,
	linkcolor={red!30!black},
	citecolor={blue!50!black},
	urlcolor={blue!80!black}
}

\newcommand{\F}{\mathbb{F}}

\newcommand{\Q}{\mathbb{Q}}

\newcommand{\C}{\mathbb{C}}
\newcommand{\Cc}{\mathbb{C}^{\times}}

\newcommand{\Ql}{\bar{\mathbb{Q}}_{\ell}}

\newcommand{\cF}{\mathcal{F}}

\newcommand{\cO}{\mathcal{O}}
\newcommand{\bb}{\mathbf{b}}

\newcommand{\bg}{\mathbf{g}}

\newcommand{\bt}{\mathbf{t}}

\newcommand{\bB}{\mathbf{B}}
\newcommand{\bG}{\mathbf{G}}
\newcommand{\bL}{\mathbf{L}}
\newcommand{\bT}{\mathbf{T}}

\newcommand{\bH}{\mathbf{H}}

\newcommand{\bU}{\mathbf{U}}

\newcommand{\fg}{\mathfrak{g}}

\newcommand{\fb}{\mathfrak{b}}
\newcommand{\ft}{\mathfrak{t}}

\newcommand{\fm}{\mathfrak{m}}

\newcommand{\sgn}{\operatorname{sgn}}
\newcommand{\rank}{\operatorname{rank}}

\newcommand{\ad}{\operatorname{ad}}

\newcommand{\IC}{\operatorname{IC}}

\newcommand{\Res}{\operatorname{Res}}

\newcommand{\cci}{C_c^{\infty}}

\newcommand{\tn}{\mathrm{tn}}

\newcommand{\Hom}{\operatorname{Hom}}

\newcommand{\Aut}{\operatorname{Aut}}

\newcommand{\tr}{\operatorname{tr}}

\newcommand{\SO}{\mathrm{SO}}
\newcommand{\OO}{\mathrm{O}}
\newcommand{\Sp}{\mathrm{Sp}}
\newcommand{\Spin}{\mathrm{Spin}}
\newcommand{\GL}{\mathrm{GL}}
\newcommand{\SL}{\mathrm{SL}}
\newcommand{\SU}{\mathrm{SU}}

\newcommand{\fsl}{\mathfrak{sl}}
\newcommand{\fso}{\mathfrak{so}}
\newcommand{\fsp}{\mathfrak{sp}}
\newcommand{\fspin}{\mathfrak{spin}}
\newcommand{\fsu}{\mathfrak{su}}
\newcommand{\Fr}{\mathrm{Fr}}

\newcommand{\nil}{\mathrm{nil}}
\newcommand{\st}{\mathrm{st}}
\newcommand{\FC}{\operatorname{FC}}
\newcommand{\Bred}{\mathcal{B}_{\mathrm{red}}}
\usepackage{leftindex}
\newcommand{\li}{\leftindex}

\newcommand{\qd}[2]{\left(\frac{#1}{#2}\right)}

\usepackage{mathrsfs}
\newcommand{\scC}{\mathscr{C}}

\usepackage{aliascnt}
\theoremstyle{plain}
\newtheorem{theorem}{Theorem}
\newaliascnt{lemma}{theorem}
\newaliascnt{proposition}{theorem}
\newaliascnt{corollary}{theorem}
\newaliascnt{remark}{theorem}
\newaliascnt{definition}{theorem}
\newaliascnt{conjecture}{theorem}
\newaliascnt{example}{theorem}
\newaliascnt{question}{theorem}
\newtheorem{lemma}[lemma]{Lemma}

\newtheorem{corollary}[corollary]{Corollary}

\theoremstyle{remark}
\newtheorem{remark}[remark]{Remark}

\aliascntresetthe{lemma}
\aliascntresetthe{proposition}
\aliascntresetthe{corollary}
\aliascntresetthe{remark}
\aliascntresetthe{definition}
\aliascntresetthe{example}
\aliascntresetthe{conjecture}
\aliascntresetthe{question}

\title{Lusztig constants and endoscopy}
\author{Wille Liu}
\address{Institute of Mathematics, Academia Sinica, 6F, Astronomy-Mathematics Building, No. 1, Sec. 4, Roosevelt Road, Taipei, Taiwan}\email{wliu@gate.sinica.edu.tw}
\author{Wei-Hsuan Hsin}
\address{Department of Mathematics, National Taiwan University, Astronomy-Mathematics Building 5F, No. 1, Sec. 4, Roosevelt Road, Taipei, Taiwan}\email{b11201030@ntu.edu.tw}
\author{Cheng-Chiang Tsai}\address{Institute of Mathematics, Academia Sinica, 6F, Astronomy-Mathematics Building, No. 1, Sec. 4, Roosevelt Road, Taipei, Taiwan,\vskip.05cm
\noindent also Department of Applied Mathematics, National Sun Yat-Sen University, and Department of Mathematics, National Taiwan University}\email{chchtsai@gate.sinica.edu.tw}
\thanks{W.-H. H. is supported by NSTC grant 114-2628-M-001-003. C.-C. T. is supported by NSTC grant 114-2115-M-001-009, 114-2628-M-001-003, and AS grant AS-IV-115-M07.}

\begin{document}

\begin{abstract}
We prove that on a semisimple Lie algebra $\mathfrak{g}$ over a finite field of large characteristic, if a complex-valued invariant function $f$ and its Fourier transform $\hat f$ are both supported in the nilpotent cone of $\mathfrak{g}$, then $\hat f = \gamma^{-1}f$ for an explicit quadratic Gauss sum $\gamma$. Consequently, we determine a fourth root of unity appearing in various formulae of generalised Gel'fand--Graev characters, known as Lusztig constant, previously known in special cases due to works of Kawanaka, Digne--Lehrer--Michel, Waldspurger and Geck. As consequence, we show the validity of a conjecture of Letellier on the compatibility of Fourier transform with Deligne--Lusztig induction.
\end{abstract}

\maketitle

\section{Introduction}\label{sec:intro}

Let $\bG$ be a connected semisimple algebraic group defined over a finite field $\F_q$ and let $\bg$ denote its Lie algebra. 

We assume there exists a non-degenerate $\bG$-invariant symmetric bilinear form $\langle\relbar,\relbar\rangle$ on $\bg$ defined over $\F_q$ and fix such a choice. Let $G = \bG(\F_q)$ and $\fg = \bg(\F_q)$ denote their $\F_q$-points. Let $C(\fg)^G$ be the space of $\C$-valued functions on $\fg$ that are invariant under the adjoint $G$-action.
Fix a non-trivial additive character $\psi:\F_q\to \Cc$. The Fourier transform for a function $f\in C(\fg)$ is defined by 
\[
    \hat f\in C(\fg),\quad \hat f(X) = q^{-\dim \fg/2}\sum_{Y\in\fg} \psi(\langle X,Y\rangle)f(Y).
\]
It satisfies the involutivity $\hat{\hat{f}}(X) = f(-X)$ and preserves the subspace $C(\fg)^G$ of $G$-invariant functions.
The \emph{Weil index} of $\fg$ is defined to be the normalised Gauss sum
\[
    \gamma_{\psi}(\fg) = q^{-\dim \fg/2}\sum_{X\in \fg}\psi\left(\frac{\langle X,X\rangle}{2}\right).
\]
It is known to be a fourth root of unity. 
Following Waldspurger, we define
\[
\FC(\fg) = C(\fg^{\nil})^G\cap \widehat{C(\fg^{\nil})^G}\subseteq C(\fg)^G
\]
to be the subspace of invariant functions $f\in C(\fg)^G$ such that both $f$ and $\hat f$ vanish ouside the nilpotent cone $\fg^{\nil} = \bg^{\nil}(\F_q)\subseteq\fg$. The main result of this paper, proven in~\autoref{sec:proof}, is the following:
\begin{theorem}\label{thm:main}
Assume that $p > 3(h(\bG)-1)$, where $h(\bG)$ is the supremum of the Coxeter numbers of the simple\footnote{Throughout this article, `simple' means `absolutely quasi-simple'.} factors of $\bG$. For $f\in \FC(\fg)$, we have
\[
    \hat f = \gamma_{\psi}(\fg)^{-1}f.
\]
\end{theorem}
The fourth root of unity $\zeta$ for which $\hat f = \zeta f$, called \emph{Lusztig constant} (for a cuspidal block) in~\cite{Let05}, has appeared in the study of characters of finite reductive groups, notably in the calculation of generalised Gel'fand--Graev characters in Kawanaka~\cite{Kaw86} and Lusztig~\cite{Lus92}. Explicit formulae for $\zeta$ are previously known for $\bG$ simple of types $G_2$, $F_4$ and $E_8$ by Kawanaka~\cite{Kaw86} (see also Geck~\cite{Gec99}), $A_n$ by Digne--Lehrer--Michel~\cite{DLM92,DLM97}, and classical groups $\SO(2n+1), \Sp(2n)$ and $\SO(2n)$ by Waldspurger~\cite{Wal01}. We calculate $\gamma_{\psi}(\fg)$ in~\autoref{ssec:Weil-simple} for simple groups $\bG$ and explain how \autoref{thm:main} is compatible with these formulae. 
\par
The proof of~\autoref{thm:main} proceeds by induction on $\dim \bG$. The base cases are the trivial group $\bG = 1$ and types $G_2$, $F_4$ and $E_8$, where the formula is known due to Kawanaka~\cite{Kaw86}. The inductive step relies mainly on Waldspurger's works~\cite{Wal22,Wal25}, which describe the endoscopic transfer of the space $\FC(\fg)$ for a $p$-adic Lie algebra $\fg$, as well as the commutation of Fourier transform with endoscopic transfer up to Weil indices, proved in~\cite{Wal95,Wal97} and conditioned on the validity of the fundamental lemma for Lie algebras, later verified by Ng\^o~\cite{Ngo10}. \par
In~\autoref{sec:let}, we deduce two conjectures of Letellier~\cite{Let05} concerning the change of Lusztig constants across different $\F_q$-forms on $\bG$ and the compatibility of Fourier transform with Deligne--Lusztig induction.

\section{Weil indices and Lusztig constants}
We fix throughout this article an isomorphism $\overline{\Q}_\ell \cong \C$. We also fix a finite field $\F_q$ be a finite field of odd characterstic $p$.
\subsection{Fourier transform and Weil indices}
Choose a non-trivial additive character $\psi: \F_q\to \Cc$. Given a finite-dimensional $\F_q$-vector space $V$ equipped with a non-degenerate symmetric bilinear form $Q:V\times V\to \F_q$. The Fourier transform is defined by
\[
(f\mapsto \hat f) : C(V)\to C(V^*), \quad \hat f(X) = q^{-\dim V / 2}\sum_{Y\in V} f(Y)\psi(Q(X, Y))
\]
It satisfies the involutivity: $\hat{\hat f}(X) = f(-X)$. \par
The \emph{Weil index} for the pair $(V,Q)$ is defined to be the normalised Gauss sum:
\[
    \gamma_{\psi}(V,Q) = q^{-\dim V/2}\sum_{X\in V}\psi\left(\frac{Q(X,X)}{2}\right).
\]
See~\cite[appendix A]{Ran93} for reference.
We will sometimes suppress $\psi$ or $Q$ or both from the notation when it is clear from the context. It is known that $\gamma_\psi(V,Q)$ is a fourth root of unity. Let
\[
\varepsilon(\psi) = q^{-1/2}\sum_{X\in \F_q} \psi(X^2)
\]
denote the normalised classical Gauss sum. It is well-known that
\[
    \gamma_{\psi}(V,Q) = \varepsilon(\psi)^{\dim V}\left(\frac{\det(Q/2)}{q}\right),
\]
where $\left(\frac{\bullet}{q}\right)$ is the quadratic character (Legendre symbol) of $\F_q^{\times}$. \par

Let $\sigma\in \OO(V,Q)$ be a linear automorphism of $V$ preserving $Q$. We introduce the Galois twist $(\li^{\sigma}V, \li^{\sigma}Q)$ as follows: the quadratic space $(V, Q)$ extends by linearity to $(\overline{V}, \overline{Q})$ to the algebraic closure $\overline{\F}_q$; let $\li^{\sigma}V = \overline{V}^{\sigma\Fr}$ be the subspace of $(\sigma\Fr)$-invariants; since $\overline{Q}$ is $(\Fr, \sigma)$-invariant, it descends by restriction to a symmetric bilinear form $\li^{\sigma}Q$ on $\li^{\sigma}V$. 
\begin{lemma}\label{lem:galois-twist}
The Weil indices of $(V, Q)$ and $(\li^{\sigma}V, \li^{\sigma}Q)$ are related as follows:
\[
\gamma_{\psi}(\li^{\sigma}V, \li^{\sigma}Q) = \det(\sigma)\gamma_{\psi}(V, Q).
\]
\end{lemma}
\begin{proof}
Since the Lang map on $\GL(\overline{V})$ is surjective, we may choose $A\in \GL(\overline{V})$ such that $A^{-1}\Fr(A) = \sigma\in \GL(\overline{V})$. Choose a basis $\{v_1, \ldots, v_d\}$ for $V$. In this basis, we express $A$ as a matrix $A = (a_{i,j})_{i,j=1}^d$ and we express the bilinear form $Q$ as a matrix $M = (Q(v_i,v_j))_{i,j=1}^d$. Then, the bilinear form $\li^\sigma Q$ is expressed by the matrix $A^tMA$ in the basis $\{Av_1, \ldots, Av_d\}$ of $\li^\sigma V$. It follows that
\[
    \gamma_{\psi}(\li^{\sigma}V, \li^{\sigma}Q)\gamma_{\psi}(V, Q)^{-1} = \qd{\det(A^tMA/2)/\det(M/2)}{q} = \qd{\det(A)^2}{q}.
\]
On the other hand, since $\det(A)^{-1}\Fr(\det A) = \det(\sigma)$, we have $\det(A)\in \F_q^{\times}$ if and only if $\det(\sigma) = 1$. Therefore, $\qd{\det(A)^2}{q} = \det(\sigma)$.
\end{proof}

\subsection{Weil indices for reductive Lie algebras}
We resume to the setup of~\autoref{sec:intro}. Let $\Fr:\bG\to \bG$ and $\Fr: \bg\to \bg$ denote the Frobenius maps, so that $G = \bG^{\Fr}$, $\fg = \bg^{\Fr}$ are the points. Choose a non-trivial additive character $\psi:\F_q\to \Cc$ and a $(\bG,\Fr)$-invariant non-degenerate bilinear form $\langle\relbar,\relbar\rangle_{\bg}$ on $\bg$, which induces an isomorphism $\bg\xrightarrow{\sim} \bg^*$. Taking $\Fr$-invariants, we get a non-degenerate bilinear form $\langle\relbar,\relbar\rangle_{\fg}$ on $\fg$. The Fourier transform $(f\mapsto \hat f): C(\fg)\to C(\fg)$ preserves the $G$-invariant functions and yields $(f\mapsto \hat f):C(\fg)^G\to C(\fg)^G$. We will denote by $\gamma(\fg) = \gamma_{\psi}(\fg)$ the Weil index of $\fg$ equipped with the symmetric bilinear form $\langle\relbar,\relbar\rangle_{\fg}$. \par

We explain now how to calculate the Weil index $\gamma(\fg)$. The following lemma is easly to prove:
\begin{lemma}\label{lem:reduction-weil}\leavevmode
\begin{enumerate}
\item
If $\varphi:\bG'\to \bG$ is an isogeny defined over $\F_q$ and $\bg'$ is equipped with the is the induced bilinear form $\langle\relbar,\relbar\rangle_{\bg'} = \langle\relbar,\relbar\rangle_{\bg}\circ(d\varphi\times d\varphi)$, then $\gamma(\fg) = \gamma(\fg')$.
\item
If $\bG = \bG_1\times \cdots\times \bG_r$, and $\bg_i$ is equipped with the bilinear form $\langle\relbar,\relbar\rangle_{\bg}|_{\bg_i\times\bg_i}$ for $i = 1, \ldots, r$, then $\gamma(\fg) = \gamma(\fg_1)\cdots\gamma(\fg_r)$.
\item
If $\bG = \Res_{\F_{q^r}/\F_q}\bG_0$ and $\langle\relbar,\relbar\rangle_{\bg} = \tr_{\F_{q^r}/\F_q}\circ\langle\relbar,\relbar\rangle_{\bg_0}$, then $\gamma_{\psi}(\bg) = \gamma_{\psi_0}(\bg_0)$, where $\psi_0 = \psi\circ\tr_{\F_{q^r}/\F_q}:\F_{q^r}\to \Cc$.
\item
For $a\in \F_q^{\times}$, we have
\[
    \gamma(\fg, a\langle\relbar,\relbar\rangle) = \gamma(\fg, \langle\relbar,\relbar\rangle)\left(\frac{a}{q}\right)^{\dim \fg}.
\]\qed
\end{enumerate}
\end{lemma}
\begin{remark}
If $\bG$ is simple of even rank, (4) implies that $\gamma(\fg)$ is independent of the choice of $\langle\relbar,\relbar\rangle_{\bg}$. Changing $\langle\relbar,\relbar\rangle_{\bg}$ can result in a change of sign for $\gamma(\fg)$ in general.
\end{remark}
Since every connected reductive group over a finite field is quasi-split, we can choose a $\Fr$-stable Borel pair $(\bB, \bT)$ for $\bG$. Let $\bb$ and $\bt$ denote the Lie algebras and $B, T, \fb, \ft$ the $\Fr$-fixed points of the respective groups and Lie algebras.
\begin{lemma}\label{lem:weil-G-T}\leavevmode
\begin{enumerate}
\item
The restriction $\langle \relbar,\relbar\rangle\mapsto \langle \relbar,\relbar\rangle|_{\bt\times\bt}$ induces a bijection between the set of $(\bG,\Fr)$-invariant bilinear forms on $\bg$ and that of $(N_{\bG}(\bT), \Fr)$-invariant bilinear forms on $\bt$.
\item
Let $\ft$ be equipped with the bilinear form $\langle \relbar,\relbar\rangle|_{\ft\times \ft}$. Then, we have
\[
    \gamma(\fg) = \gamma(\ft).
\]
\end{enumerate}
\end{lemma}
\begin{proof}
Statement (i) is standard. Let $\bU$ denote the unipotent radical of $\bB$ and $\bU^-$ that of the opposite Borel $\bB^-$. Then, statement (ii) follows from the observation that $\mathfrak{u}\oplus\mathfrak{u}^-$ is a sum of copies of the hyperbolic plan, whose Weil index equals $1$.
\end{proof}
It is therefore enough to calculate $\gamma(\ft)$ for simply connected simple groups $\bG$ defined over $\F_q$ with any choice of non-degenerate $N_{\bG}(\bT)$-invariant bilinear form on $\bt$. 
\subsection{Quasi-split twisting}
Let $(\bB, \bT, \mathbf{x})$ be a pinned Borel pair for a simply connected semisimple Chevalley group $\bG$ defined over $\F_q$.  Let $\theta$ be an automorphism of the Dynkin diagram of $\bG$, which induces an automorphism of $(\bG,\bB, \bT, \mathbf{x})$. Define $(\li^\theta\bG, \li^\theta\bB, \li^\theta\bT) = (\bG, \bB, \bT)$ with the $\F_q$-structure twisted by $\theta$ and let $\li^{\theta}\Fr = \theta\circ\Fr$ denote the Frobenius morphism for $\li^{\theta}\bG$. Assume that $\bg$ is equipped with a $(\bG,\Fr)$-invariant symmetric bilinear form $\langle \relbar,\relbar\rangle$ which is also $\theta$-invariant. It induces via the identity $\bg = \li^\theta\bg$ a $(\li^{\theta}\bG,\li^{\theta}\Fr)$-invariant bilinear form $\langle\relbar,\relbar\rangle_{\li^\theta\bg}$ on $\li^\theta\bg$. 
\begin{lemma}\label{lem:quasi-split-twist}
We have
\[
\gamma(\li^\theta\fg) =  \sgn(\theta)\gamma(\fg),
\]
where $\sgn(\theta)\in\{\pm 1\}$ is the sign of $\theta$ as a permutation.
\end{lemma}
\begin{proof}
It follows immediately from~\autoref{lem:galois-twist}.
\end{proof}

\subsection{Weil indices for simple Lie algebras}\label{ssec:Weil-simple}
Via~\autoref{lem:weil-G-T}, the calculation of Weil indices for simple Lie algebras is reduced to the maximal torus $\bt$ of a $\Fr$-stable Borel pair. We shall choose a non-degenerate invariant symmetric bilinear form on the simple Lie algebra of each type. \par
For the split group $\SL(n)$, we choose $(2n)$-times the trace pairing: $(X, Y)\mapsto 2n\tr(XY)$. For the split groups $\Spin(2n+1), \Sp(2n)$ and $\Spin(2n)$, we choose the trace pairing $(X, Y)\mapsto \tr(XY)$. Then, in these cases
\[
    \gamma(\fsl(n)) = \varepsilon(\psi)^{n-1},\;\gamma(\fspin(2n+1)) = \gamma(\fsp(2n)) = \gamma(\fspin(2n)) = \varepsilon(\psi)^n.
\]
For the quasi-split groups $\SU_{\F_{q^2}/\F_q}(n)$ and $\Spin_{\F_{q^2}/\F_q}(2n)$, we apply \autoref{lem:quasi-split-twist} to $\fsl(n)$ and $\fspin(2n)$ respectively and get
\[
    \gamma(\fsu_{\F_{q^2}/\F_q}(n)) = (-1)^{\lfloor \frac{n-1}{2}\rfloor}\varepsilon(\psi)^{n-1},\; \gamma(\fspin_{\F_{q^2}/\F_q}(2n)) = -\varepsilon(\psi)^n.
\]
For type $G_2$, the symmetrised Cartan matrix is
\[
Q = \begin{pmatrix}2 & -1 \\ -3 & 2\end{pmatrix}\begin{pmatrix}1 & 0 \\ 0 & 3\end{pmatrix} = \begin{pmatrix}2 & -3 \\ -3 & 6\end{pmatrix}.
\]
Then, 
\[
    \langle X,Y\rangle = \sum_{i,j} Q_{i,j}\omega_i(X)\omega_j(Y)
\]
defines a $N_{\bG}(\bT)$-invariant non-degenerate bilinear form on $\bt$. It follows that
\[
    \gamma(G_2) = \left(\frac{\det (Q/2)}{q}\right)\varepsilon(\psi)^{\operatorname{rank}G_2} = \left(\frac{3}{q}\right)\left(\frac{-1}{q}\right) = \left(\frac{q}{3}\right).
\]
The same method defines a bilinear form and allows to compute the Weil indices for $E_{6,7,8}$ and $F_4$:
\[
        \gamma(E_6)  = \left(\frac{q}{3}\right),\; \gamma(E_7)  = \varepsilon(\psi)^{-1},\;  \gamma(E_8) = \gamma(F_4) = 1.
\]
Finally,~\autoref{lem:quasi-split-twist} yields
\[
    \gamma(\li^2E_6) =  \gamma(E_6) = \left(\frac{q}{3}\right),\; \gamma(\li^3D_4) =  \gamma(\fspin(8)) = 1.
\]
\begin{remark}\label{rem:E8F4G2}
The above formulae for $\gamma(E_8)$, $\gamma(F_4)$ and $\gamma(G_2)$ agree with those for the sign $\varepsilon$ in~\cite[3.3.10]{Kaw86} (see also~\cite[3.6]{Gec99}). The formula for $\gamma(\fsl(n))$ agrees with the one for $\zeta_{\mathscr{I}}^{-1}$ in~\cite[p. 141]{DLM97} in the case $e=n$, up to a misprint from them --- their formula in the case where $e$ is even should read
\[
\zeta_{\mathscr{I}}^{-1} = q^{\frac{n(1-e)}{2e}}((-1)^{\frac{q-1}{2}} q)^\frac{n(e-2)}{2e}\mathscr{G}(\varepsilon)^{\frac{n}{e}},
\]
and the right-hand side equals $\varepsilon(\psi)^{n(e-1)/e}$ in our notation. As for $\fg = \fsp(k(k+1)/2), \fso(k^2)$ for $k\ge 1$, the formulae for $\gamma(\fg)$ agree with the formulae for $\li^o\gamma^{-1}$ in~\cite[V.8]{Wal01} --- note that $\gamma(\fso(k^2)) = \gamma(\fspin(k^2))=1$ when $k$ is odd since the rank $n=(k^2-1)/2$ is divisible by $4$ in this case; similarly, the rank $n=k^2/2$ is even if $k$ is even.
\end{remark}

\subsection{The space \texorpdfstring{$\FC(\fg)$}{FC(g)}}\label{ssec:FC}
Let $\FC(\fg)$ be as defined in~\autoref{sec:intro}. If $Z(\bg) \neq 0$, we have clearly $\FC(\fg) = 0$, so let us assume here that $\bg$ is semisimple. For every $\Fr$-stable ($\Ql$-)cuspidal pair $(\cO, \scC)$ on $\bg$ (in the sense of~\cite{Lus84}), we fix a Weil structure $\Fr_{\scC}:\Fr^* \scC\xrightarrow \to \scC$ and consider the characteristic function 
\[
\chi_{\scC}\in C(\fg),\; \chi_{\scC}(X) = \begin{cases}0 & \text{if $X\in \fg\setminus \cO^{\Fr}$,} \\ \tr(\Fr_{\scC}, \scC_{X}) & \text{if $X\in \cO^{\Fr}$.}\end{cases}
\]
Then, we have $\chi_{\scC}\in \FC(\fg)$, the line $\C\chi_{\scC}$ is independent of the choice of $\Fr_{\scC}$, and $\{\chi_{\scC}\}_{(\cO,\scC)}$ form a basis of $\FC(\fg)$ when $(\cO,\scC)$ runs over the $\Fr$-stable cuspidal pairs on $\bg$ up to isomorphism. 

\subsection{Lusztig constants}\label{ssec:lusztig}
Given an $\Fr$-stable cuspidal pair $(\cO, \scC)$ equipped with a Weil structure $\Fr^*\scC\xrightarrow{\sim} \scC$ as in~\autoref{ssec:FC}, the Fourier--Deligne transform $\operatorname{FD}(\IC(\scC))$ comes with the induced Weil structure. By the Fourier-invariance of $\IC(\scC)$, and since any two Weil structures on $\IC(\scC)$ are proportional, we have
\[
    \widehat{\chi_{\scC}} = \zeta(\scC)\chi_{\scC}
\]
for some constant $\zeta(\scC)\in \Cc$, called the \emph{Lusztig constant} of $(\cO, \scC)$. It is obvious that $\zeta(\scC)$ does not depend on the choice of Weil structure on $\scC$. By the involutivity of Fourier transform, we see that $\zeta(\scC)^4 = 1$. 
We leave the proof of the following easy lemma to the reader:
\begin{lemma}\label{lem:reduction-lusztig}\leavevmode
\begin{enumerate}
\item 
If $\alpha:\bG'\to \bG$ is an isogeny and $(\cO', \scC') = \alpha^*(\cO,\scC)$ is the induced cuspidal pair, then $\zeta(\scC') = \zeta(\scC)$.
\item
If $G = G_1\times \cdots\times G_r$ and $(\cO,\scC) \cong (\cO_1,\scC_1)\boxtimes \cdots\boxtimes (\cO_r,\scC_r)$, then 
\[
\zeta(\scC) = \prod_{i = 1}^r \zeta(\scC_i).
\]
\item
If $\bG = \Res_{\F_{q^r}/\F_q}\bG'$, $\langle\relbar,\relbar\rangle_{\bg} = \tr_{\F_{q^r}/\F_q}\circ\langle\relbar,\relbar\rangle_{\bg'}$ and $(\cO, \scC)$ is induced from a $\Fr$-stable cuspidal pair $(\cO',\scC')$ on $\bG'$ by the Weil restriction, then $\zeta_{\psi}(\scC) = \zeta_{\psi'}(\scC')$, where $\psi' = \psi\circ\tr_{\F_{q^r}/\F_q}:\F_{q^r}\to \Cc$.
\end{enumerate}\qed
\end{lemma}

\section{Weil indices and Lusztig constants for \texorpdfstring{$p$}{p}-adic Lie algebras}
Let $F$ be a non-archimedean local field of characteristic $0$ with residue field $\F_q$. Let $\cO_F\subset F$ be the ring of integers and $\fm_F\subset \cO_F$ the maximal ideal. Let $F^{\mathrm{alg}}$ be an algebraic closure of $F$ and $\Gamma_F = \operatorname{Gal}(F^{\mathrm{alg}}/F)$ the absolute Galois group of $F$.
\subsection{Fourier transform and Weil indices}
Choose an additive character $\psi: F\to \Cc$ of conductor $\fm_F$ (namely, $\psi(\fm_F) = 1$ and $\psi|_{\cO_F}\neq 1$). It induces a character $\psi_0:\F_q\to \Cc$ by restriction to $\cO_F$ and reduction modulo $\fm_F$. Given a $F$-vector space $V$ and $\cO_F$-lattices $V_{\cO}^+\subseteq V_{\cO}\subseteq V$ such that $V_{\cO}^+\subseteq \fm_F V_{\cO}$, we normalise the Haar measure $d_V$ on $V$ such that $d_V(V_{\cO}) = |V_{\cO}/V^+_{\cO}|^{1/2}$. The Fourier transform is defined by
\[
(f\mapsto \hat f): C^{\infty}_c(V)\to C^{\infty}_c(V^*), \quad \hat f(X) = \int_{V}f(X)\psi(\langle X, Y\rangle) d_VY.
\]
We equip $V^*$ with the dual Haar measure, so that $\hat{\hat f}(X) = f(-X)$. Set 
\[
    V_{\cO}^{\perp} = \{X\in V^*\;;\; X(V_{\cO})\subseteq \fm_F\}. 
\]
Then, the quotient $(V^+_{\cO})^{\perp}/(V_{\cO})^{\perp}$ can be identified with the $\F_q$-linear dual of $V_{\cO} / V^+_{\cO}$. The normalisation of Haar measure is made so that the Fourier transform commutes with the inclusions $C(V_{\cO}/V^+_{\cO}) \subset C^{\infty}_c(V)$ and $C((V_{\cO}/V^+_{\cO})^*) \subset C^{\infty}_c(V^*)$. \par
Let $Q:V\times V\to F$ be a non-degenerate symmetric bilinear form. Upon choosing a Haar measure $d_V$ on $V$ and an $\cO_F$-lattice $V_{\cO}\subseteq V$ which is co-isotropic with respect to $B$, namely $(V_{\cO})^{\perp}\subseteq V_{\cO}$, we set
\[
    G_{\psi}(V_{\cO},Q) = \int_{V_{\cO}} \psi\left(\frac{Q(X,X)}{2}\right) d_VX.
\]
Then, we have $|G_{\psi}(V_{\cO},Q)| = (d_V(V_{\cO})d_V((V_{\cO})^{\perp}))^{1/2}$ and the quotient
\[
    \gamma_{\psi}(V, Q) = |G_{\psi}(V_{\cO},Q)|^{-1}G_{\psi}(V_{\cO},Q) 
\]
is independent of the choice of the co-isotropic $\cO_F$-lattice $V_{\cO}$. The number $\gamma_{\psi}(V, Q)$ is called the \emph{Weil index} of $(V,Q)$. See~\cite[Appendix A]{Ran93} and~\cite[VIII.1]{Wal95} for references. We will sometimes suppress $\psi$ or $Q$ or both from the notation.
\subsection{\texorpdfstring{$p$}{p}-Adic Lie algebras}\label{ssec:p}
Let $\bG$ be a connected reductive group defined over $F$. We assume that $p \ge \sup(2h(\bG) +1, \rank \bG)$, so that $\bG$ splits over a finite extension of $F$ of $p$-prime order, which is thus tame. We will also denote denote $\bG = \bG(F^{\mathrm{alg}})$ the group of $F^{\mathrm{alg}}$-points and $G = \bG(F) = \bG^{\Gamma_F}$ the group of $F$-points. Similarly, we let $\bg$ denote the Lie algebra of $\bG$ and set $\fg = \bg^{\Gamma_F}$. 

Let $\Bred(G)$ denote the reduced Bruhat--Tits building of $G$. Given any facet $\cF\subset\Bred(G)$, let $\fg_{\cF,0}$ and $\fg_{\cF,0+}$ denote respectively the parahoric subalgebra of $\fg$ and its radical associated with $\cF$. We consider $(\bG,\Gamma_F)$-invariant symmetric bilinear forms $\langle\relbar,\relbar\rangle$ on $\bg$ satisfying the following property: 
\begin{equation}\label{eq:normalised}
(\fg_{\cF,0})^{\perp}  = \fg_{\cF,0+} \quad \text{for every $\cF\in \Bred(G)$} \tag{*}.
\end{equation}
\begin{lemma}\label{lem:descent}
The property~\eqref{eq:normalised} is stable under Galois extension and Galois descent: if $E/F$ is a finite Galois extension, then $\langle\relbar,\relbar\rangle$ satisfies~\eqref{eq:normalised} for $\fg$ if and only if it satisfies~\eqref{eq:normalised} for $\fg_E = \bg(E)$. 
\end{lemma}
\begin{proof} Thanks to \cite[Proposition 4.1]{AR00} and its proof by reducing to the split case, a form $\langle\relbar,\relbar\rangle$ satisfying ~\eqref{eq:normalised} for all finite Galois extension $E/F$ exist. Our assumption on $p$ ensures that $\bg$ is the product of the simple $F$-factors. The property ~\eqref{eq:normalised} pins down $\langle\relbar,\relbar\rangle$ (if exists) on each simple factor up to a scalar in $\cO_F$, and therefore any other choice of $\langle\relbar,\relbar\rangle$ satisfies ~\eqref{eq:normalised} for one $E/F$ iff it does for all $E/F$.
\end{proof}
When $\bG$ is simple, the set of $(\bG,\Gamma_F)$-invariant symmetric bilinear form $\langle\relbar,\relbar\rangle$ on $\bg$ satisfying~\eqref{eq:normalised} is an $\cO_F^{\times}$-torsor, and each form induces a non-degenerate $(\bG_{\cF},\Fr)$-invariant symmetric bilinear form on the reductive Lie algebra $\bg_{\cF}$ over $\F_q$ associated with any facet $\cF\subset \Bred(G)$: given an unramified Galois extension $E/F$ of degree $r\ge 1$, let $G_E = \bG(E)$, $\fg_E = \bg(E)$ and $\cF_E \subset \Bred(G_E)$ the unique facet containing $\cF$; then $\fg_{E,\cF_E,0+}$ is isotropic with respect to the bilinear form $\langle \relbar,\relbar\rangle|_{\fg_E\times\fg_E}$ and $(\fg_{E,\cF_E,0+})^{\perp} = \fg_{E,\cF_E,0}$; it follows that $\langle \relbar,\relbar\rangle|_{\fg_E\times\fg_E}$ induces by modulo $\fm_E$ a bilinear form on the quotient space $\fg_{E,\cF_E,0}/\fg_{E,\cF_E,0+} \cong \bg_{\cF}(\F_{q^r})$; letting $r$ run over positive integers, we obtain a bilinear form on $\bg_{\cF}$ satisfying the required properties.
\begin{lemma}\label{lem:normalised-torus}
Given a maximal subtorus $\bT$ of $\bG$ defined over $F$, the restriction 
\[
\langle\relbar,\relbar\rangle\mapsto \langle\relbar,\relbar\rangle|_{\bt\times \bt}
\]
induces a bijection from the set of $(\bG,\Gamma_F)$-invariant symmetric bilinear forms on $\bg$ satisfying~\eqref{eq:normalised} to the set of $(N_\bG(\bT),\Gamma_F)$-invariant symmetric bilinear forms on $\bt$ satisfying~\eqref{eq:normalised}.
\end{lemma} 
\begin{proof}
When $\bT$ is a split torus, we can choose an integral model of $\bG$ such that $\bT\subset\bG$ is also a maximal subtorus over $\cO_F$. There is the (unique) point $s_T\in\Bred(T)$ and a corresponding hyperspecial vertex $s\in\Bred(G)$ so that, $\ft_{s_T,0}=\ft(\cO_F)$, $\ft_{s_T,0+}=\ker(\ft(\cO_F)\to\ft(\F_q))$, $\fg_{s,0}=\fg(\cO_F)$ and $\fg_{s,0+}=\ker(\fg(\cO_F)\to\fg(\F_q))$. The result then follows from \autoref{lem:weil-G-T} and the fact that both sets are $\cO_F^{\times}$-torsors. In general, the statement can be reduced to the split case by~\autoref{lem:descent}. 
\end{proof}

Let $\bg$ be equipped with a $(\bG,\Gamma_F)$-invariant symmetric bilinear form $\langle\relbar,\relbar\rangle_{\bg}$ satisfying~\eqref{eq:normalised}. Taking the $\Gamma_F$-invariants, we obtain a $G$-invariant symmetric bilinear form $\langle\relbar,\relbar\rangle_{\fg}$ on $\fg$ and define the Weil index $\gamma_{\psi}(\fg)$ with respect to $\langle\relbar,\relbar\rangle_{\fg}$. 

\begin{lemma}\label{lem:reduction-Weil-p}\leavevmode
Let $\cF\subset \Bred(G)$ be a facet and let $\fg_{\cF}$ be equipped with the invariant symmetric bilinear form induced from $\langle\relbar,\relbar\rangle_{\bg}$ by reduction modulo $\fm_F$. Then we have 
\[
    \gamma_{\psi}(\fg) = \gamma_{\psi_0}(\fg_{\cF}).
\]
\end{lemma}
\begin{proof}
This follows from a known property of Weil indices, see~\cite[Satz 3]{Car64}.
\end{proof}
Let $d_{\fg}$ be the self-dual Haar measure on $\fg$, so that $d_{\fg}(\fg_{\cF, 0}) = q^{\dim \bg_s/2}$. The Fourier transform on $\fg$ is defined by
\[
(f\mapsto \hat f): \cci(\fg)\to \cci(\fg),\quad \hat f(X) = \int_{\fg} \psi(\langle X, Y\rangle) f(Y)d_{\fg} Y.
\]
It satisfies $\hat{\hat f}(X) = f(-X)$ and descends to a map $I(\fg)\to I(\fg)$. Given any facet $\cF\subset \Bred(G)$, we have a reductive group $\bG_{\cF}$ defined over $\F_q$. Given a function $f\in C(\fg_{\cF})$, we denote by $f_{\cF}\in \cci(\fg)$ the function on $\fg$ obtained by pulling back $f$ along the quotient map $\fg_{\cF, 0}\to \fg_{\cF,0}/\fg_{\cF,0+} = \fg_{\cF}$ and extending by $0$ outside $\fg_{\cF,0}$. Then, we have 
\begin{equation}\label{eq:fourier-induction}
    \widehat{f_{\cF}} = (\hat f)_{\cF}
\end{equation} 
for every $f\in C(\fg_{\cF})$.
\subsection{The space \texorpdfstring{$\FC(\fg)$}{FC(g)}}\label{ssec:FC-p}
Let $\cci(\fg)$ denote the space of $\C$-valued compactly supported locally constant functions on $\fg$ and $I(\fg):=\cci(\fg)_G$ the $G$-coinvariant quotient space. Let $\fg^{\tn}\subseteq \fg$ denote the subset of topologically nilpotent elements. Define $I(\fg^{\tn})$ to be the image of $\cci(\fg^{\tn})$ in $I(\fg)$.  Choose a $(\bG,\Gamma_F)$-invariant symmetric bilinear form $\langle\relbar,\relbar\rangle_{\bg}$ satisfying~\eqref{eq:normalised}. The Fourier transform on $\fg$ descends to the $G$-coinvariant $I(\fg)$ and yields a map $(f\mapsto \hat f):I(\fg)\to I(\fg)$. We set 
\[
    \FC(\fg) = I(\fg^{\tn})\cap \widehat{I(\fg^{\tn})}. 
\]
If the torus $Z(\bG)^{\circ}$ is not totally ramified over $F$, then $\FC(\fg) = 0$. We assume henceforth that $Z(\bG)^{\circ}$ is totally ramified over $F$. In particular the reduced building is the same as the extended building. \par
Let $V(\Bred(G))\subseteq \Bred(G)$ be the set of vertices of $\Bred(G)$. Given any $s\in V(\Bred(G))$, the space $\FC(\fg_s)$ defined in~\autoref{ssec:FC} associated with $\bG_s$ can be regarded as a subspace of $\cci(\fg)$ via the inclusion $f\mapsto f_s$. Let $\chi_{s,\scC}$ denote the image in $I(\fg)$ of the function $(\chi_{\scC})_s\in \cci(\fg)$ obtained from $\chi_{\scC}\in \FC(\fg_s)$ in this way. By~\eqref{eq:fourier-induction}, we have 
\begin{equation}\label{eq:reduction-lusztig}
    \widehat{\chi_{s, \scC}} = \zeta(s,\scC)\chi_{s, \scC}, 
\end{equation}
where $\zeta(s,\scC) = \zeta(\scC)$ is the Lusztig constant for the $\Fr$-stable cuspidal pair $(\cO, \scC)$ as defined in~\autoref{ssec:lusztig}. We let $\FC(\fg)_s$ denote the subspace of $\FC(\fg)$ spanned by $\{\chi_{s,\scC}\}_{(\cO, \scC)}$ as $(\cO,\scC)$ runs over the $\Fr$-stable cuspidal pairs on $\bg_s$. Then, it is shown in~\cite{Wal22} that there is a decomposition
\begin{equation}\label{eq:decomposition-s}
    \FC(\fg) = \bigoplus_{[s]\in V(\Bred(G))/G}\FC(\fg)_s,
\end{equation}
and $\{\chi_{s,\scC}\}_{s, (\cO, \scC)}$ forms a basis of $\FC(\fg)$ as $s$ runs over a set of representatives of the $G$-orbits in $V(\Bred(G))$ and $(\cO, \scC)$ runs over the $\Fr$-stable cuspidal pairs on $\fg_s$ up to isomorphism. \par
Set $\Omega = G_{\ad} / \pi(G)$, where $\pi:G\to G_{\ad}$ is induced from the canonical map $\bG\to \bG_{\ad}$, and set $\Xi = \Hom(\Omega, \Cc)$. The actions of $G_{\ad}$ on $\fg$ induces an action of $\Omega$ on $I(\fg)$ which preserves the subspace $\FC(\fg)$ and induces the isotypic decomposition
\begin{equation}\label{eq:decomposition-xi}
    \FC(\fg) = \bigoplus_{\xi\in \Xi}\FC(\fg)_{\xi}.
\end{equation}
Given a $G_{\ad}$-orbit $O\subseteq V(\Bred(G))$, the $G_{\ad}$-action on $\FC(\fg)$ preserves the sum $\FC(\fg)_O = \bigoplus_{[s]\in O/G}\FC(\fg)_s$. 
Given any $\xi\in \Xi$, we form the function 
\[
\chi_{O, \scC,\xi} = \sum_{g\in \Omega} \xi(g)\chi_{gs, g_*\scC}\in \FC(\fg)_{O, \xi}.
\]
If it is non-zero, then we have 
\begin{equation}\label{eq:lusztig-O-xi}
    \widehat{\chi_{O, \scC, \xi}} = \zeta(O,\scC, \xi)\chi_{O, \scC,\xi},\quad \text{where $\zeta(O,\scC, \xi) = \zeta(s, \scC)$}.
\end{equation}
Indeed, this is because $\zeta(s,\scC) = \zeta(gs, g_*\scC)$ for every $s\in G_{\ad}$.
\subsection{Endoscopic transfer for \texorpdfstring{$\FC(\fg)$}{FC(g)}}\label{ssec:endoscopy}
Until the end of this section, we assume that $\bG$ is simply connected, simple and quasi-split, and fix a $\Gamma_F$-stable Borel pair $(\bB, \bT)$ for $\bG$. The complex dual group $\hat G$ of $\bG$ is equipped with an action of the Galois group $\Gamma_F$ and a $\Gamma_F$-stable Borel pair $(\hat B, \hat T)$ induced from $(\bB, \bT)$. Let $\hat\Delta\subset X^*(\hat T)$ denote the set of simple roots of $\hat G$ and $\hat\Pi_a = \hat\Pi\cup \{-\hat\theta\}$, where $\hat\theta$ is the largest root of $\hat G$ with respect to $(\hat B, \hat T)$. Let $\sigma \mapsto \sigma_{\bG}\in \Aut(\hat T)$ denote the $\Gamma_F$-action on $\hat T$. We let $\hat\Omega$ denote the subgroup of the Weyl group $\hat W = \hat W(\hat G, \hat T)$ formed by the elements whose action on $X^*(\hat T)$ conserves the set $\hat\Pi_a$. \par
Let $\kappa$ be an elliptic endoscopic datum for $\bG$ with endoscopic group $\bG'$. The complex dual group $\hat G'$ of $\bG'$ comes with a $\Gamma_F$-stable Borel pair $(\hat B', \hat T)$ induced from $(\hat B, \hat T)$ and satisfying $\hat\Delta(\hat B',\hat T)\subset \hat\Delta_a$, and there exists a maps $\omega_{\bG'}:\Gamma_F\to \hat\Omega$ such that $\sigma_{\bG'} = \omega_{\bG'}(\sigma)\sigma_{\bG}$ for every $\sigma\in \Gamma_F$. Thus, we have a natural isomorphism $\bT\cong \bT'$, which transports the $(N_{\bG}(\bT), \Gamma_F)$-invariant bilinear form $\langle\relbar,\relbar\rangle_{\bg}|_{\bt\times \bt}$ on $\bt$ to a $(N_{\bG'}(\bT'), \Gamma_F)$-invariant bilinear form on $\bt'$, which extends to a $(\bG',\Gamma_F)$-invariant bilinear form $\langle\relbar,\relbar\rangle_{\bg'}$ on $\bg'$.
\begin{lemma}\label{lem:normalised-endoscopy}
If $\langle\relbar,\relbar\rangle_{\bg}$ satisfies~\eqref{eq:normalised}, then $\langle\relbar,\relbar\rangle_{\bg'}$ satisfies~\eqref{eq:normalised} as well.
\end{lemma}
\begin{proof}
    This follows immediately from~\autoref{lem:descent} and~\autoref{lem:normalised-torus}.
\end{proof}

The main results of~\cite{Wal22,Wal25} show that, given a $G_{\ad}$-orbit $O\subseteq V(\Bred(G))$ a cuspidal pair $(\cO, \scC)$ on $\fg_s$ for $s\in O$ and a $\xi\in \Xi$ such that $\chi_{O, \scC,\xi}\neq 0$, there is an essentially unique elliptic endoscopic datum $\kappa$ with endoscopic group denoted by $\bG'$, such that $\chi_{O, \scC, \xi}$ is $\kappa$-stable, and $\chi_{O, \scC,\xi}$ has a unique smooth $\kappa$-endoscopic transfer in $\FC(\fg')^{\st}$, which we denote by $\chi_{O,\scC,\xi}^{\kappa}$, and moreover, $\chi_{O,\scC,\xi}^{\kappa}$ lies in $\FC(\fg')_{O',\xi'}$ for some $G'_{\ad}$-orbit $O'\subseteq V(\Bred(G'))$ and some $\xi'\in \Xi(G')$. \par
\begin{lemma}\label{lem:transfer-fourier}
Let $O$, $\xi$, $\kappa$, $O'$ and $\xi'$ be as above. Then we have
\[
    \widehat{\chi_{O,\scC,\xi}^{\kappa}} =  \zeta^{\kappa}(O, \scC,\xi)\chi_{O,\scC,\xi}^{\kappa},
\]
where $\zeta^{\kappa}(O, \scC,\xi)$ the constant satisfying
\[
    \gamma(\fg)\zeta(O, \scC,\xi) = \gamma(\fg')\zeta^{\kappa}(O, \scC,\xi).
\]
\end{lemma}
\begin{proof}
This follows from~\cite[Conjecture 1, VIII.7]{Wal95}, which is proven in~\cite[Th\'{e}or\`{e}me 1.4]{Wal97} based on the fundamental lemma for Lie algebras, available by~\cite{Ngo10}.
\end{proof}

\section{Proof of~\autoref{thm:main}}\label{sec:proof}
Let $\bG_0$ be a connected reductive group over $\F_q$, $\psi_0:\F_q\to \Cc$ a non-zero character and $\langle\relbar,\relbar\rangle_{\fg_0}$ a non-degenerate $\bG_0$-invariant symmetric bilinear form as in the statement. We prove the statement by induction on $\dim\bG_0$. When $\dim \bG_0 = 0$, we have $\FC(\fg_0) = 0$, so the statement holds trivially. Let us assume that $n > 0$ and that the statement has been proven  for all $r\ge 1$ and for all connected reductive groups $\bH_0$ defined over $\F_{q^r}$ such that $\dim \bH_0 < n$. \par
When $Z(\fg_0) \neq 0$, we have $\FC(\fg_0) = 0$, so there is nothing to prove. We assume hence $\bG_0$ semisimple. By~\autoref{lem:reduction-weil} and~\autoref{lem:reduction-lusztig}, we are reduced to considering $\bG_0$ simple and simply connected. We have seen in~\autoref{rem:E8F4G2} that the statement is reduced to~\cite[3.3.10]{Kaw86} if $\bG_0$ is of type $G_2$, $F_4$ or $E_8$. If $\bG_0$ is of type $\li^3D_4$, there is no cuspidal pair on $\fspin(8)$, so $\FC(\fspin(8)) = 0$ and there is nothing to prove. Therefore, we assume that $\bG_0$ is simple and simply connected of one of the types $A_n$, $B_n$, $C_n$, $\li^{1,2}D_n$, $\li^{1,2} E_6$ and $E_7$,. \par
Let $F$ be any non-archimedean local field of characteristic $0$ with residue field $\F_q$ and let $\bG$ be an unramified quasi-split simple group over $F$ such that there exists an isomorphism $\bG_{s} \cong \bG_0$ for some hyperspecial vertex $s\in V(\Bred(G))$. Pick a character $\psi:F\to \Cc$ of conductor $\fm_F$ such that $\psi$ induces $\psi_0$ by reduction modulo $\fm_F$. Pick a $\bG$-invariant symmetric bilinear form $\langle\relbar,\relbar\rangle_{\bg}$ on $\bg$ satisfying~\eqref{eq:normalised}. Moreover, we may rescale $\langle\relbar,\relbar\rangle_{\bg}$ by some scalar $a\in \cO^{\times}_F$ such that the bilinear form on $\bg_0$ induced from $\langle\relbar,\relbar\rangle_{\bg}$ coincides with the bilinear form $\langle\relbar,\relbar\rangle_{\bg_0}$ chosen from the start. Let $O\subset V(\Bred(G))$ denote the $G_{\ad}$-orbit of $s$. \par

Let $(\cO, \scC)$ be a $\Fr$-stable cuspidal pair on $\bg_0$. Choose $\xi\in \Xi(G)$ such that the element $\chi_{s,\scC}\in \FC(\fg)$ has a non-zero $\FC(\fg)_{O, \xi}$-component under the decomposition~\eqref{eq:decomposition-xi}. Then, the element $\chi_{O, \xi}$ is $\kappa$-stable for some elliptic endoscopic datum $\kappa$ with endoscopic group $\bG'$. Its Lie algebra $\bg'$ is endowed with a $\bG'$-invariant symmetric bilinear form $\langle\relbar,\relbar\rangle_{\bg'}$ as in~\autoref{ssec:endoscopy}, which satisfies~\eqref{eq:normalised} by~\autoref{lem:normalised-endoscopy}. By a careful examination of~\cite[\S 4, \S 6]{Wal25}, we see that $\kappa$ is never the principal endoscopic datum, namely $\bG'\neq \bG$, under our assumption on the type of $\bG_0$, and we have thus $\dim \bG' < \dim \bG$. Then, we have for every $s'\in V(\Bred(G'))$
\[
    \dim \bG'_{s'} \le \dim \bG' < \dim \bG = \dim \bG_{s},
\]
so the induction hypothesis applies to the group $\bG'_{s'}$, which implies 
\[
    \hat f = \gamma(\fg'_{s'})^{-1} f\quad \text{for $f\in \FC(\fg'_{s'})$.} 
\]
From~\eqref{eq:decomposition-s},~\autoref{lem:reduction-Weil-p},~\eqref{eq:reduction-lusztig} and the last equation, we deduce
\begin{equation}\label{eq:induction}
    \hat f = \gamma(\fg')^{-1} f\quad \text{for $f\in \FC(\fg')$.}
\end{equation}
Consider the smooth $\kappa$-endoscopic transfer $\chi^{\kappa}_{O, \xi}\in \FC(\fg')$. 
We obtain
\[
    \gamma(\fg_{0})\zeta(\scC) \stackrel{\text{\autoref{lem:reduction-Weil-p}+\eqref{eq:reduction-lusztig}}}{=}  \gamma(\fg)\zeta(s,\scC) \stackrel{\text{\eqref{eq:lusztig-O-xi}}}{=} \gamma(\fg)\zeta(O, \scC,\xi)
    \stackrel{\text{\autoref{lem:transfer-fourier}}}{=} \gamma(\fg')\zeta^{\kappa}(O, \scC,\xi)   \stackrel{\text{\eqref{eq:induction}}}{=}  1.
\]
Namely, we have $\widehat{\chi_{\scC}} = \gamma(\fg_{0})^{-1}\chi_{\scC}$. Since $\{\chi_{\scC}\}_{(\cO, \scC)}$ spans $\FC(\fg_0)$ as $(\cO, \scC)$ runs over the $\Fr$-stable cuspidal pairs on $\bg_0$, this completes the inductive step and therefore the proof of the statement. \qed

\section{On two conjectures of Letellier}\label{sec:let}

Resume to the setup of~\autoref{sec:intro}.
In \cite{Let05}, Letellier studies the relation between the Fourier transform and Deligne--Lusztig induction $\mathcal R_L^G:C(\mathfrak l)^L \to C(\mathfrak g)^G$ on the space of invariant functions on Lie algebras. \cite[p.2, Conjecture (*)]{Let05} asserts that
$$ \widehat{(\mathcal R_L^G f)} = \varepsilon_{\bG}\varepsilon_{\bL} \mathcal R_L^G \hat f \qquad \forall\ f \in C(\mathfrak l)^L, $$
where $\varepsilon_{\bG}=(-1)^{\rank_{\F_q}(\bG)}$ and $\varepsilon_{\bL}=(-1)^{\rank_{\F_q}(\bL)}$. Letellier reduces \cite[Theorem 6.2.16]{Let05} this conjecture to some cases of another conjecture: Suppose $(\cO,\scC)$ is a $\Fr$-stable cuspidal pair on $\bg$. We have $\zeta(\scC)$ the Lusztig constant in \autoref{ssec:lusztig}. Let $\theta$ be a pinned automorphism of $\bG$, and $^{\theta}\bG$ the group over $\F_q$ with the twisted Frobenius. Suppose $(\cO,\scC)$ is also $\li^\theta\Fr$-stable. Then we can consider the Lusztig constant $\li^{\theta}\zeta(\scC)$ associated to $\scC$ and $\li^\theta\Fr$.

\begin{corollary}\label{cor:Letellier} Assume that $p > 3(h(\bG) - 1)$. \cite[p.3, Conjecture (**)]{Let05} is true, namely $\varepsilon_{\bG}\cdot\zeta(\scC)=\varepsilon_{^{\theta}\bG}\cdot\li^{\theta}\zeta(\scC)$. Therefore \cite[p.2, Conjecture (*)]{Let05} is also true.
\end{corollary}

\begin{proof} By \autoref{thm:main}, we are reduced to the observation $\varepsilon_{\bG}\cdot\gamma_{\psi}(\fg)^{-1}=\varepsilon_{^{\theta}\bG}\cdot\gamma_{\psi}(^{\theta}\fg)^{-1}$. This follows from \autoref{lem:quasi-split-twist} and that the determinant of $\theta$ on $X^*(\bT)$ is equal to $\varepsilon_{\bG}\cdot\varepsilon_{\li^\theta\bG}$ by counting the $\theta$-orbits in the Dynkin diagram of $\bG$. 
\end{proof}

However, we would like to point out that Letellier's original argument already suffices to prove \cite[p.2, Conjecture (*)]{Let05} when $p> 3(h(\bG) - 1)$. It was observed in \cite[\S6.2.18]{Let05} that a case-by-case analysis using results in \cite{Kaw86}, \cite{DLM97}, \cite{Wal01} and \cite{DLM03} resolves all but one case of \cite[p.3, Conjecture (**)]{Let05} necessary for proving \cite[p.2, Conjecture (*)]{Let05}. The last case is $\bG=\Spin(2n)$ (possibly non-split) and the cuspidal pair does not descend to $\SO(2n)$.
Such a situation does not in fact occur, as shown in the lemma below. 

\begin{lemma}
    Let $(\cO, \scC)$ be a cuspidal pair on $\Spin(2n)$ that does not arise from one on $\SO(2n)$. Then there is exactly one $\mathbb F_q$-structure on $\Spin(2n)$ for which $(\cO, \scC)$ is stabilized by the Frobenius endomorphism.
\end{lemma}

\begin{proof}
The statement follows immediately from the classification of Frobenius-stable cuspidal pairs, see~\cite[\S\S 2.6,\,2.7]{Wal25}. We provide a proof here.
    By the classification of cuspidal pairs in \cite{Lus84}, we know that $n \not=1,2,4$. 
    In particular, $\Spin(2n)$ admits exactly two $\mathbb F_q$-structures up to isomorphism, the split form and the non-split form. \par
    
    Let $\chi$ be the central character associated with $(\cO, \scC)$. Since $(\cO, \scC)$ does not arise from $\SO(2n)$, we have $\chi(\varepsilon)=-1$, 
    where $\varepsilon$ denotes the generator of 
    $$ \ker(\Spin(2n) \to \SO(2n))\cong\mu_2. $$
    Let $z_1, z_2 \in Z(\Spin(2n))$ be the two remaining elements distinct from $1$ and $\varepsilon$. Then $z_1 = \varepsilon z_2$, and hence $\chi(z_1) \neq \chi(z_2)$. On the other hand, among the split and non-split $\mathbb F_q$-structures on $\Spin_{2n}$, exactly one induces a non-trivial Frobenius action on the center $Z(\Spin(2n))$, swapping $z_1$ and $z_2$. Since $(\cO, \scC)$ is stabilized by the Frobenius endomorphism if and only if its central character $\chi$ is, it follows that there is a unique $\mathbb F_q$-structure for which $(\cO, \scC)$ is stabilized by the Frobenius endomorphism.
\end{proof}

\begin{remark} The similar proof works with a simple group of $E_6$ and a simple group of type $A$. 
\end{remark}

\section*{Acknowledgements}
We thank Emmanuel Letellier for communicating with us the status of his conjectures. 
\bibliographystyle{amsalpha}
\bibliography{biblio.bib}
\end{document}